\documentclass{article}
\usepackage{amsmath}
\usepackage{amssymb}
\usepackage{amsthm}
\usepackage{multicol}

\usepackage{subcaption} 


\usepackage[latin1]{inputenc}
\usepackage[T1]{fontenc}
\usepackage{play}
\usepackage{booktabs}
\usepackage{longtable} 
\usepackage{array} 
\usepackage{graphicx}

\usepackage{tikz}

\newtheorem{Theorem}{Theorem}

\newtheorem{corollary}[Theorem]{Corollary}

\newtheorem{lemma}{Lemma}

\newtheorem{proposition}{Proposition}
\newtheorem{example}{Example}

\begin{document}
\title{{Symmetries of Random Partitions}}
\author{Alexander Gnedin\\ {\small Queen Mary University of London}}

\maketitle
\begin{abstract}
\noindent
This paper is motivated by a recent result of Pitman and Yakubovich~\cite{PY} stating that a partially exchangeable, stationary (PES) random partition of $\mathbb{N}$ is exchangeable. This echoes an earlier theorem of Kallenberg~\cite{Kallenberg} on the equivalence of spreadability (contractability) and exchangeability for infinite partitions. We revise the hierarchy of symmetries with a focus on partitions of finite sets $[n]$, and ask about the extent to which these relaxed symmetry types differ from exchangeability, framing the question in terms of the geometry of the polytope of distributions defined by the symmetry constraints. We show that single-orbit exchangeable partitions remain extreme among spreadable and PES distributions, and that $n=5$ is the first case where non-exchangeable extreme partitions occur, causing the associated polytopes to deviate from a simplex structure.
\end{abstract}
\noindent
De Finetti-type theorems play a central role in connecting random combinatorial structures on finite sets $[n]:=\{1,\ldots,n\}$ with their projective limits.
In the setting of partitions, Kingman's representation theorem \cite{Kingman} characterises an exchangeable partition of ${\mathbb N}$ as a unique mixture of basic paintbox processes, which appear as very precise analogues of finite occupancy models.
Pitman \cite{PitmanPTRF} introduced a larger and more flexible class of partially exchangeable partitions which account for the standard ordering of blocks by their starters; in his refinement of the paintbox representation, the exchangeable case is distinguished by invariance under size-biased permutation of the sequence of block frequencies.
Both types of random symmetry fall under the overarching concept of sufficiency \cite{DF, Dynkin}, which in the combinatorial context amounts to a consistent in $n$ division of the ensemble of labelled configurations into strata where the conditional distribution is uniform.
This line has been extended to many other combinatorial structures including ordered partitions \cite{GRCS}, Fibonacci words \cite{Goodman}, partitions with constraints \cite{GCEP}, random permutations with symmetries \cite{G11, GGorinRec, GGorinPer, Kerov} and random orders \cite{GO, Jacka}.
Paintbox representations for partitions and ordered partitions invariant under block deletion are found in \cite{GHP, RCS, selfsimilar}.

A parallel and closely related framework of central measures on the path space of a Bratteli diagram emerged in the Vershik--Kerov work on the asymptotic representation theory of the symmetric groups, where a peculiar algebraic technique was developed to identify the indecomposable distributions by a multiplicative property \cite{VK}. For instance, in the scenario iconic for probabilists, the Bratteli diagram is the two-dimensional lattice (Pascal graph), a central measure corresponds to a binary exchangeable sequence, and the multiplicativity translates as independence. The monograph \cite{BO} gives a comprehensive overview of the field.

A very different kind of random symmetry known as spreadability (or contractability) appears as similarity of substructures of the same size. 
In some cases, notably for weak orders \cite{GO, Jacka}, the spreadability fits in the sufficiency framework, but for random sequences or arrays no obvious stratification of the configuration space exists.
Despite this, the Ryll-Nardzewski theorem on infinite sequences and Kallenberg's theorem on partitions of ${\mathbb N}$ establish that for these instances the spreadability is equivalent to exchangeability, see \cite{Kallenberg}.
A similar equivalence with exchangeability was shown by Pitman and Yakubovich \cite{PY} for the class of partially exchangeable stationary (PES) partitions of $\mathbb N$, namely those that are partially exchangeable in the sense of Pitman \cite{PitmanPTRF} and have the additional feature of invariance under shifts.

Beyond the cited asymptotic equivalence, little is known about how spreadability and the PES-symmetry look from the perspective of finite partitions, and how much a given symmetry type differs from exchangeability. For finite structures, both spreadability and stationarity can be realised as invariance under elementary deletion operations, analogous to the projection induced by deletion of the largest element $n$. From a geometric point of view, the distribution of a random partition of $[n]$ is a point in a polytope defined by linear constraints expressing a symmetry property. 
For exchangeable and partially exchangeable partitions, the polytope is a simplex; therefore, the generic distribution admits a unique decomposition over the basic uniform distributions, one for each stratum. In the exchangeable case, a stratum coincides with an orbit of the symmetric group, which is the family of partitions with a fixed multiset of block sizes. In contrast, a non-simplicial geometry would imply that no sufficient statistic exists in principle, in which case it is not possible to characterise a random partition by a distribution on some `shape parameter', hence generating the random object of a given class by a two-step Bayesian procedure of choosing the value of the parameter from some unique distribution, then spreading the probability uniformly or in some other standard way over the stratum associated with the selected value.

In this paper, we revisit the representations of partitions with the described symmetries, adding some results on finite and infinite representations which might be new. In particular, we show that stationarity alone implies the existence of block frequencies, and that the frequency of the first block is a size-biased pick from the whole collection. Corollary~\ref{Cor1} extends a result on single-orbit spreadable sequences~\cite{KallenbergPTRF} (Lemma 2.2) to the alphabet-symmetric case.

In Sections 4 and 6, we look into the geometry of the polytopes. For two-block or $n \leq 4$ partitions, we show that both spreadability and PES-symmetry are equivalent to exchangeability, and that for every $n$ the single-orbit distributions remain indecomposable in a wider symmetry context. The transition occurs at $n=5$, when the polytopes of spreadable and PES partitions are no longer simplices. We find that for $n=5$ the PES-polytope has just three non-exchangeable extremes, while in the spreadable case there is an explosion of complexity with no visible traces of regularity.

Though the exposition is largely self-contained, we refer to \cite{BO}, \cite{Kallenberg} and \cite{CSP} for background. To avoid cumbersome notation, some proofs are carried out via examples. Mathematica software, the Apocrita HPC cluster, and AI tools were utilised to find and verify interesting examples and counter-examples, and to identify the extreme PES partitions for $n=5$.

\section{Basic definitions}
\noindent
A partition $\pi_n$ of the set $[n]:=\{1,\ldots,n\}$ is commonly written as a list of blocks arranged in increasing order of their \emph{starters}, with the elements within each block also arranged in increasing order.
The starter is the least element of the block.
We call the \emph{shape} of the partition, denoted by $s(\pi_n)$, the {\it composition} of $n$ obtained by listing the block sizes in this order;
and the \emph{ranked shape}, denoted $s^\downarrow(\pi_n)$, the {\it partition} of $n$ obtained by  rearranging the block sizes in nonincreasing order.
For instance,
$135 \mid 26 \mid 4 \mid 789$
is a partition of $[9]$ with shape
$(3,2,1,3)$ and ranked shape $(3,3,2,1)$; see Figure~\ref{CanRep} for various representations of this partition
as a table, block-tracking word and a matrix.

With a partition $\pi_n$ we associate its restrictions
$\pi_{n-1},\ldots,\pi_1,\varnothing$ obtained by successively removing the elements $n,\ldots,1$ from their blocks.
Viewed forward in time, the sequence
$\varnothing,\pi_1,\ldots,\pi_n$
represents a growth process in which, at each step, a new element is either added to an existing block or starts a new one.
For instance, in Figure~\ref{subfig:CanPart} 
the element $10$ can be either added to the top of one of the columns, or placed in the bottom row as a new starter.
We further associate with $\pi_n$ its copies on $n$-element subsets
$A\subset\mathbb N$, obtained via the increasing bijection between $[n]$ and $A$.

We may think of partition as allocation of labelled balls in a series of boxes, with understanding that boxes are only distinguishable by their content. 
In the sequential allocation framework a starter is a ball opening an empty box. 

\begin{figure}[htbp]
    \centering
    \begin{subfigure}[b]{0.45\textwidth}
        \centering
        \begin{tikzpicture}[x=0.75cm, y=0.75cm, line width=0.8pt, every node/.style={font=\small}]
            \draw (0,0) rectangle (1,3); \draw (0,1) -- (1,1); \draw (0,2) -- (1,2);
            \node at (0.5,0.5) {\textbf{1}}; \node at (0.5,1.5) {3}; \node at (0.5,2.5) {5};
            \draw (1,0) rectangle (2,2); \draw (1,1) -- (2,1);
            \node at (1.5,0.5) {\textbf{2}}; \node at (1.5,1.5) {6};
            \draw (2,0) rectangle (3,1); \node at (2.5,0.5) {\textbf{4}};
            \draw (3,0) rectangle (4,3); \draw (3,1) -- (4,1); \draw (3,2) -- (4,2);
            \node at (3.5,0.5) {\textbf{7}}; \node at (3.5,1.5) {8}; \node at (3.5,2.5) {9};
        \end{tikzpicture}
        \caption{Canonical representation as a table.}\label{subfig:CanPart}
    \end{subfigure}
    \hfill
    \begin{subfigure}[b]{0.50\textwidth}
        \centering
        \begin{tikzpicture}[x=0.65cm, y=0.65cm, line width=0.8pt, every node/.style={font=\small}]
            \foreach \x in {0,1,...,9} { \draw (\x,0) -- (\x,4); }
            \foreach \y in {0,1,...,4} { \draw (0,\y) -- (9,\y); }
            \fill (0.5, 0.5) circle (0.15); \fill (1.5, 1.5) circle (0.15); 
            \fill (2.5, 0.5) circle (0.15); \fill (3.5, 2.5) circle (0.15); 
            \fill (4.5, 0.5) circle (0.15); \fill (5.5, 1.5) circle (0.15); 
            \fill (6.5, 3.5) circle (0.15); \fill (7.5, 3.5) circle (0.15); 
            \fill (8.5, 3.5) circle (0.15); 
            \node[below] at (0.5, -0.1) {\textbf{1}}; \node[below] at (1.5, -0.1) {\textbf{2}};
            \node[below] at (2.5, -0.1) {\textbf{1}}; \node[below] at (3.5, -0.1) {\textbf{3}};
            \node[below] at (4.5, -0.1) {\textbf{1}}; \node[below] at (5.5, -0.1) {\textbf{2}};
            \node[below] at (6.5, -0.1) {\textbf{4}}; \node[below] at (7.5, -0.1) {\textbf{4}};
            \node[below] at (8.5, -0.1) {\textbf{4}};
        \end{tikzpicture}
        \caption{Block-tracking word.}\label{subfig:GridPart}
    \end{subfigure}

    \vspace{2.5em} 

    \begin{subfigure}[b]{\textwidth}
        \centering
        \scalebox{0.8}{
            $\begin{pmatrix}
                1 & 0 & 1 & 0 & 1 & 0 & 0 & 0 & 0 \\
                0 & 1 & 0 & 0 & 0 & 1 & 0 & 0 & 0 \\
                1 & 0 & 1 & 0 & 1 & 0 & 0 & 0 & 0 \\
                0 & 0 & 0 & 1 & 0 & 0 & 0 & 0 & 0 \\
                1 & 0 & 1 & 0 & 1 & 0 & 0 & 0 & 0 \\
                0 & 1 & 0 & 0 & 0 & 1 & 0 & 0 & 0 \\
                0 & 0 & 0 & 0 & 0 & 0 & 1 & 1 & 1 \\
                0 & 0 & 0 & 0 & 0 & 0 & 1 & 1 & 1 \\
                0 & 0 & 0 & 0 & 0 & 0 & 1 & 1 & 1
            \end{pmatrix}$
        }
        \caption{The equivalence relation matrix.}\label{subfig:MatrixRep}
    \end{subfigure}

    \vspace{1.5em}
    \caption{Representations of the partition of $[9]$ with shape $(3,2,1,3)$ and block starters 1,2,4,7.}
    \label{CanRep}
\end{figure}

For $j\le n$, let $D_j$ denote the operator acting on $\pi_n$ by deleting the element $j$ and then relabelling the remaining elements by $[n-1]$.
Note that $D_n$ coincides with the canonical restriction:
\[
D_n\pi_n=\pi_{n-1}.
\]
This operation is relatively simple, compared with $D_j$ for $j<n$, which may cause some reshuffling of the ordering of blocks by starters.
In particular, $D_1$ amounts to deleting ball $1$, moving 
 box  with ball $2$ to the front, and finally
re-labelling each remaining ball by the next smaller integer: for instance,
\[
D_1(135 \mid 26 \mid 4 \mid 789)=15 \mid 24 \mid 3 \mid 678,
\]
where the block $15$ is the image of $26$.

Denote by $\mathcal P_n$ the set of partitions of $[n]$, and let $\mathcal P$ be the projective limit of the system $(\mathcal P_n,D_n)$.
Thus $\pi\in\mathcal P$ is a partition of $\mathbb N$, identified with the consistent sequence of its restrictions
$\pi=(\pi_n)$ to finite sets.
It is natural to view $\mathcal P$ as an infinite rooted tree with root $\varnothing$, vertices at level $n$ labeled by $\mathcal P_n$,
and branching dictated by the relation $D_n\pi_n=\pi_{n-1}$, which determines the possible extensions of $\pi_{n-1}$ to partitions of $[n]$.
In this picture, $\pi$ is identified with an infinite rooted path which at each level moves to an adjacent vertex farther from the root.

The symmetric group ${\mathfrak S}_n$ acts on ${\mathcal P}_n$ by re-labelling the balls.
Identifying ${\mathfrak S}_n$ with the subgroup of bijections $\sigma:\mathbb N\to\mathbb N$ satisfying $\sigma(i)=i$ for all $i>n$, we obtain an action of the infinite symmetric group
${\mathfrak S}_\infty=\bigcup_{n\ge1}{\mathfrak S}_n$
on $\mathcal P$.

A random partition $\Pi_n$ of $[n]$ is a random variable with values in the
finite set $\mathcal P_n$.
The space of distributions $P$ on $\mathcal P_n$, seen as vectors of point masses $(P(\pi_n), ~\pi_n\in{\mathcal P}_n)$,     is a finite-dimensional simplex
whose extreme points are the delta measures assigning full mass to individual
partitions of $[n]$.
Passing to the restriction $\Pi_{n-1}=D_n\Pi_n$, we canonically associate with
$\Pi_n$ random partitions of smaller size, down to $\Pi_0=\varnothing$.

A random partition $\Pi=(\Pi_n)$ of $\mathbb N$ is a random variable with
values in the projective limit space  $\mathcal P$.
The Borel sigma-algebra on $\mathcal P$ is inherited from the discrete product topology.
The space of distributions on $\mathcal P$ is an infinite-dimensional simplex,
compact in the product topology, where convergence amounts 
to the marginal convergence in distribution of each $\Pi_n$.

\section{Partial exchangeability}\label{PE}

\noindent
Following  \cite{PitmanPTRF},
a  finite random partition $\Pi_n$ is {\it partially exchangeable} if its distribution
is conditionally uniform given the shape $s(\Pi_n)$.
That is, there exists a {\it probability function} $p$ on the set of compositions of $n$
such that
$$
{\mathbb P}[\Pi_n=\pi_n]=p(s(\pi_n)), \qquad \pi_n\in{\mathcal P}_n.
$$
It should be noted that the value, say $p(3,2,1,3)$, is not the probability that
the shape $(3,2,1,3)$ occurs, but rather the probability of any particular partition $\pi_n\in{\mathcal P}_n$
with this shape.
For a composition $\lambda=(\lambda_1,\ldots,\lambda_k)$ of $n$, 

\begin{equation}\label{Pp}
{\mathbb P}[s(\Pi_n)=\lambda] =d(\lambda)p(\lambda),
\end{equation}
where the combinatorial factor (sometimes called the dimension) counts the number of partitions of $[n]$ with shape $\lambda$:
\begin{equation}\label{dim-pe}
d(\lambda)
=
\prod_{j=1}^{k-1}
{\Lambda_j-1\choose\lambda_j-1},
\qquad
\Lambda_j:=\lambda_j+\cdots+\lambda_k.
\end{equation}

Although expanding (\ref{dim-pe}) leads to a massive cancellation of factors, this form is best suitable 
to describe the random partition associated with the {\it elementary} probability function $p_\lambda$, characterised by
\[
d(\lambda)p_\lambda(\lambda)=1,
\]
which corresponds to the uniform distribution over partitions of $[n]$ with shape $\lambda$.
To that end, it will be convenient to view a partition of $[n]$ as an allocation of labelled balls into a series of boxes $1,\ldots,k$ with fixed capacities $\lambda_1,\ldots,\lambda_k$,
where box $i$ is subdivided into $\lambda_i$ cells, as represented by columns in Figure \ref{subfig:CanPart}.
There are two equivalent constructions of the random partition under $p_\lambda$:
\begin{enumerate}
\item[(i)] {\it Box by box.}
Ball $1$ opens the first box, and then $\lambda_1-1$ of the remaining balls are chosen uniformly at random and placed into the vacant cells of that box.
Among the $\Lambda_2$ balls remaining, the one with the least label opens the next box, and then another $\lambda_2-1$ balls are chosen uniformly at random and placed into that box, and so on.

\item[(ii)] {\it Ball by ball.}
For $0\le m<n$, after allocating the first $m$ balls, suppose the currently open boxes contain altogether $r$ vacant cells.
Then ball $m+1$ fills one of these cells, chosen uniformly at random, with probability $r/(n-m)$;
otherwise, with probability $(n-m-r)/(n-m)$, it opens the next box.
\end{enumerate}

Considerably more insight is gained by realising partially exchangeable partitions
within the richer structure of random permutations, whose distribution is
conditionally uniform given certain multivariate statistic, which we choose here to be 
 the {\it set of lower record values}.
For this purpose we write a permutation in the one-line notation
$\tau(1)\tau(2)\cdots\tau(n)$,
interpreted as  {\it ranking} of the set of $n$ balls by the integers $1,\ldots,n$.
A pair $(i,\tau(i))$ is a lower record (with record value $\tau(i)$) if
$\tau(i)<\tau(j)$ for $1\le j<i$.

Fix a composition
\[
\lambda=(\lambda_1,\ldots,\lambda_k),
\qquad
\Lambda_j:=\lambda_j+\cdots+\lambda_k,
\qquad
\Lambda_{k+1}:=0,
\]
and consider a random permutation of $[n]$ with uniform distribution over the set
of permutations whose lower record values are fixed at
\[
\Lambda_2+1>\Lambda_3+1>\cdots>\Lambda_{k+1}+1.
\]
Define block $j$ to be the set of balls $i$ whose rank satisfies
\[
\Lambda_{j+1}<\tau(i)\le \Lambda_j,
\qquad j=1,\ldots,k.
\]
The resulting ordered partition of the set of balls has probability function $p_\lambda$.
The corresponding distribution of the random permutation differs from \eqref{Pp}
only by an inessential combinatorial factor, namely the number of ways to order
the nonstarter balls within each box.
The starters (record positions in permutation) form a random sequence
\[
1=\rho_1<\rho_2<\cdots<\rho_k\le n,
\]
where $\rho_j$ is characterised by the condition
\[
\tau(\rho_j)=\Lambda_{j+1}+1.
\]

By this realisation the distribution of the starters admits a simple closed form.
For
\[
1=r_1<r_2<\cdots<r_k\le n,
\]
under $p_\lambda$
\begin{equation}\label{RecDis}
\mathbb P[\rho_1=r_1,\ldots,\rho_k=r_k]
=
\lambda_1
\prod_{j=2}^{k}
\frac{\lambda_j(r_j-r_{j-1}-1)!}
     {(\Lambda_j-r_j+r_{j-1})_{\,r_j-r_{j-1}}},
\end{equation}
where
$(a)_m=a(a+1)\cdots(a+m-1)$
denotes the rising factorial.
Indeed, for each $j\ge2$, the balls
$i\in[r_{j-1}+1,r_j)$
must have ranks avoiding the interval
$(\Lambda_{j+1},\Lambda_j]$,
while ball $r_j$ has rank $\Lambda_{j+1}+1$, the next lower record.

By a result of Nacu \cite{Nacu}, formula \eqref{RecDis} extends by linearity to 
bijection between the simplex of partially exchangeable partitions of $[n]$
and the simplex of distributions of the starters. Note, however, that while
the distribution on shapes may be arbitrary, this is not so for the starters:
for instance, delta measures are excluded.

In relation to the group action, partial exchangeability amounts to invariance under re-labellings that do not alter the shape. More precisely,
\[
{\mathbb P }[\Pi_n=\sigma\pi_n]={\mathbb P }[\Pi_n=\pi_n], \qquad \sigma\in{\mathfrak S}_n,
\]
provided
$s(\sigma\pi_n)=s(\pi_n).$
Here $\sigma$ and $\pi_n$ are, in general, not independent: a permutation $\sigma$ preserves shape only for some partitions, and for a given $\pi_n$ there is a subgroup of permutations with this property.
In terms of the representations in Figure \ref{CanRep}, always admissible are permutations of columns that do not alter the (strict, upper) record values $1,2,\ldots$ in the block-tracking sequence.
Among other possibilities are the adjacent transpositions $(a, a+1)$ of starters whose boxes contain equal number of balls.


A random infinite partition $\Pi=(\Pi_n)$ is partially exchangeable if every
$\Pi_n$ is partially exchangeable.
This definition is consistent because partial exchangeability is a
{\it hereditary} property: if $\Pi_n$ is partially exchangeable, then its
restriction $\Pi_{n-1}=D_n\Pi_n$ is partially exchangeable as well.
The reason is combinatorial: if two partitions have the same shape $\lambda$,
then for every composition $\mu$ on a higher level both have the same number
$d(\lambda,\mu)$ of extensions with shape $\mu$.
Algebraically, this turns the consistency of distributions into the linear
recursion on the probability function,
\begin{equation}\label{p-Rec}
p(\lambda)=\sum_\mu d(\lambda,\mu)p(\mu),
\end{equation}
where the summation runs over compositions $\mu$ on any fixed level higher
than that of $\lambda$.
The set of (nonnegative) solutions to (\ref{p-Rec}) satisfying the 
 total probability condition $p(\varnothing)=1$ is an infinite-dimensional simplex,
where each element has a unique representation as a convex mixture of extreme points.

In the infinite system (\ref{p-Rec}) it is sufficient to leave only equations with the next level expansion.
Then, for $\lambda=(\lambda_1,\ldots,\lambda_k)$ a composition of $n$ and $\mu$ a
composition of $n+1$, the only nonzero coefficients are $d(\lambda,\mu)=1$
if $\mu_i=\lambda_i+1$ for some $i\le k$, or $\mu_{k+1}=1$, with all other
parts in $\mu$ same as in $\lambda$.
Such a single extension has a natural interpretation in terms of ranking permutations:
a new ball $n+1$ is appended to the string $\tau(1)\cdots\tau(n)$ with an arbitrary rank 
$1\leq\tau(n+1)\leq n+1$, and the
first $n$ balls are re-ranked accordingly. For the updates of the set of lower record values
this induces the usual interlacing pattern.

The following process constructs an infinite partially exchangeable partition
from a randomly marked sequence of balls $1,2,\ldots$; the starters receive special marks, and the other balls will be marked in i.i.d. fashion.
Let
\[
1\ge L_1\ge L_2\ge\cdots\ge0
\]
be an arbitrary infinite nonincreasing random sequence, independent of another
sequence $U_1,U_2,\ldots$ of i.i.d.\ $[0,1]$-uniform random variables.
Conditionally on $L_1,L_2,\ldots$ we define a new sequence
$V_1,V_2,\ldots$ combining both sequences.
At step $1$ we set $V_1=L_1$.
Inductively, if at step $n>1$ we have
\[
\min(V_1,\ldots,V_{n-1})=L_k,
\]
then we set
\[
V_n=
\begin{cases}
U_n, &\text{if } U_n>L_k,\\
L_{k+1}, &\text{if } U_n\le L_k,
\end{cases}
\]
with mark $U_n$ discarded in the second case.
Eventually the whole $L$-sequence will be included in the $V$-sequence, with
the exception of the terminal zeroes that appear in case $L_k=0$ for some $k>1$.
From this,
the partition  $\Pi$ is constructed by letting $n$ be a starter if
$V_n=L_k$ for some $k\ge1$, and letting each $m>n$ belong to the block of $n$
if
\[L_{k-1}>V_m>L_k .\]
Then
\[F_k:=L_{k-1}-L_k,\qquad L_0:=1,\]
is the asymptotic frequency of the block corresponding to $L_k$, which is the block of the $k$th starter.

Observe that $L_1,L_2,\ldots$ appear as (weak, i.e. with ties possible) lower record values in the
$V$-sequence, and that each $L_k>0$ replaces one element of the $U$-sequence.
Thus conditionally on $L_1,L_2,\ldots$ the $V$-sequence may be regarded as an
i.i.d.\ uniform sequence conditioned on the values of lower records.
Straight from the construction, the probability function corresponding to the
choice of distribution of $L_1,L_2,\ldots$ is
\begin{equation}\label{PErep}
p(\lambda_1,\ldots,\lambda_k)
=
{\mathbb E}\prod_{i=1}^k
L_{i-1}\,F_i^{\,\lambda_i-1},
\end{equation}
and it is readily checked that this $p$ indeed satisfies \eqref{p-Rec}.

The extreme solutions to (\ref{p-Rec}) correspond  to deterministic $L$-sequences.
A good exercise in the extreme case is to scrutinise the consequences of any of the relations
\begin{enumerate}
\item[(a)] $L_1=1$,
\item[(b)] $L_1=0$,
\item[(c)] $L_k=L_{k+1}>0$ for some $k>0$,
\item[(d)] $1>L_1>L_2>\cdots>L_k=0$,
\item[(e)] $\lim_{k\to\infty}L_k>0$, with the subcase that the sequence is eventually constant.
\end{enumerate}

The main structural result on infinite partially exchangeable partitions 
asserts that \eqref{PErep} sets up a bijection between probability functions
and distributions of $(L_1,L_2,\ldots)$. The first proof, due to Pitman \cite{PitmanPTRF},
used de Finetti--type arguments based on the exchangeability of non-starters.
Kerov \cite{Kerov} employed 
an explicit formula for the kernel
$d(\lambda,\mu)/d(\mu)$ to identify all limit points of the set of elementary
probability functions (a Martin boundary), and then applied the seminal
Kerov--Vershik `Ring Theorem' to identify these limits with the extreme
solutions to \eqref{p-Rec}.
Interpretations in terms of records and permutations, as well as connections
with infinite graphs, appeared in \cite{Kerov} and were further developed for
symmetries most closely related to partial exchangeability in
\cite{GCEP, G11}.

\section{Exchangeability} 

\noindent
A partially exchangeable random partition $\Pi_n$ of $[n]$ is {\it exchangeable}
if its probability function $p$ is symmetric, so the distribution factors through the ranked shape as
\begin{equation}\label{Fin-Ex}
{\mathbb P}[\Pi_n=\pi_n]
=
p\!\left(s^\downarrow(\pi_n)\right),
\qquad \pi_n\in{\mathcal P}_n.
\end{equation}
Equivalently,
\begin{enumerate}
\item[(i)] $\Pi_n\stackrel{d}{=}\sigma\Pi_n,\qquad \sigma\in{\mathfrak S}_n$;
\item[(ii)] for every partition $\lambda$ of $n$, conditionally on
$s^\downarrow(\Pi_n)=\lambda$, the distribution of $\Pi_n$ is uniform over
the set of partitions of $[n]$ with ranked shape $\lambda$.
\end{enumerate}
Property (i)  expresses invariance under the action of ${\mathfrak S}_n$, and is the most common definition of exchangeable partition of $[n]$.
Since adjacent transpositions  $(a, a+1)$ generate the whole symmetric group, invariance under
swapping two elements already implies (i). Swapping a starter with any other element does not change the distribution.
Property (ii) identifies the ranked shape as a sufficient statistic whose values
enumerate  the orbits of
this group action. For instance, the uniform distribution on $\mathcal P_n$ is exchangeable with constant probability function 
\begin{equation}\label{Bell}
\widehat{p}(\lambda):=\frac{1}{B_n}
\end{equation}
 on compositions of $n$, 
where $B_n$ is the Bell number.

Compared with partial exchangeability,
 the difference is that the order in which the boxes are opened is no longer tied to the occupancy counts.
To highlight this distinction, we invoke
two sequential constructions of the {\it elementary symmetric} probability function $p_\lambda^*$, corresponding to the 
uniform distribution over partitions of $[n]$ with
 ranked shape $\lambda$.
 Consider a series of boxes $1,\ldots,k$ with positive capacities $\lambda_1\geq \cdots\geq \lambda_k$, where
$\lambda_1+\cdots+\lambda_k=n$.
\begin{enumerate}

\item[(i)] {\it Ball by ball.}
At stage $m$, ball $m$ is placed uniformly at random into one of the remaining vacant cells.
\item[(ii)] {\it Box by box.}
Ball $1$ chooses a cell uniformly at random, thereby opening some random box $\tau(1)$ containing that cell, and then $\lambda_{\tau(1)}-1$ of the remaining balls are chosen uniformly at random and placed into that box.
Then, among the $n-\lambda_{\tau(1)}$ balls remaining, the one with the least label chooses one of the remaining vacant cells uniformly at random, thereby opening the next box $\tau(2)$, and another $\lambda_{\tau(2)}-1$ of the remaining balls are chosen uniformly at random and placed into that box, and so on.
\end{enumerate}
The random arrangement of boxes $\tau(1),\ldots, \tau(k)$ emerging this way is a {\it size-biased} permutation $\tau$ of $[k]$ directed by $\lambda$, in particular ${\mathbb P}[\tau(1)=i]=\lambda_i/n$.  For the identity permutation 
$${\mathbb P}[\tau={\rm id}]=\prod_{j=1}^k \frac{\lambda_j}{\Lambda_j},~~~\quad\quad~~~\Lambda_j=\lambda_j+\cdots+\lambda_k,$$
and the general formula is obtained  from this   by relabelling the boxes.  
The size-biased invariance property of this distribution means that  $\tau'\tau \stackrel{d}{=}\tau$ for independent replica $\tau'$ of $\tau$.
The arrangement of block sizes by increase of the starters, $\tau\lambda:=(\lambda_{\tau(1)},\ldots,\lambda_{\tau(k)})$, has therefore the property of invariance under the size-biased permutation, that is
$\tau'\tau\lambda\stackrel{d}{=}\tau\lambda$.

Conditionally on $\tau$  the allocation follows the elementary probability function $p_{\tau\lambda}$, with shape $\tau\lambda$
Now, understanding $\tau$ as the generic value of this random permutation, we obtain
a decomposition
\begin{equation}\label{symmetris}
p_\lambda^*=\sum_{\tau\in{\mathfrak S}_k} w(\tau)p_{\tau\lambda},
\end{equation}
where the weights comprise the  distribution of the size-biased paermutation. 
The number of distinct terms in the symmetrisation (\ref{symmetris}) depends on repetitions in the multiset $\{\lambda_1,\ldots,\lambda_k\}$.
 In particular, if (and only if) all parts of $\lambda$ are equal, the distribution of $\tau$ is uniform,  $w(\tau)\equiv 1/k!$, and $p_\lambda$ is itself symmetric.

Recall that every permutation $\rho\in{\mathfrak S}_n$  admits a unique decomposition in disjoint cycles.
Relabelling the elements of $\rho$ by another permutation $\sigma\in{\mathfrak S}_n$ amounts to conjugation $\sigma\rho\sigma^{-1}$ which preserves the ordered shape 
of the partition in cycles. Therefore a distribution on ${\mathfrak S}_n$ invariant under conjugations yields an exchangeable partition. Conversely, given exchangeable $\Pi_n$,
such a random permutation emerges by arranging balls within boxes in independent uniformly random cyclic orders.
To illustrate the connection, the distribution on ${\mathfrak S}_n$ assigning probability $\theta^{k}/(\theta)_n$ to every permutation $\rho$ with $k$ cycles corresponds to the 
seminal Ewens distribution on ${\mathcal P}_n$ with parameter $\theta\in[0,\infty]$.
The projective limit of this measure-preserving actions of $\mathfrak S_n$ on  ` $\mathfrak S_n$ as enriched $\mathcal P_n$'  plays an important role 
in the representation theory of the  infinite symmetric group \cite{BO}.

An infinite partition $\Pi=(\Pi_n)$ is exchangeable if \eqref{Fin-Ex} holds
for every $n$, equivalently if
$\Pi\stackrel{d}{=}\sigma\Pi$ for every $\sigma\in{\mathfrak S}_\infty$.
Since exchangeability implies partial exchangeability,
\eqref{PErep} remains valid, although the sequence $(L_k)$ can no longer have
an arbitrary distribution.
The representation is more naturally restated in terms of the sequence of  nonnegative block
frequencies
\[
F_k=L_{k-1}-L_k, ~k\geq 1,\qquad F_0:=1-\sum_{k=1}^\infty F_k,
\]
which satisfies the additional condition of invariance under size-biased
permutation.

In Kingman's representation via
 the ranked sequence 
\[
F_1^\downarrow\ge F_2^\downarrow\ge\cdots,
\]
the product in \eqref{PErep} is symmetrised, and
the resulting 
 infinite series is the monomial symmetric function  in the variables
$(F_k^\downarrow)$,
extended to include the possibility $F_0>0$.
The random variable $F_k^\downarrow$ appears as the limiting frequency of the
$k$th largest block of $\Pi_n$.
The extreme exchangeable distributions on ${\mathcal P}$ correspond to
deterministic sequences $(F_k^\downarrow)$.
It is a further instructive exercise to identify how conditions
(a)--(e) from Section~\ref{PE} appear in the exchangeable case.

A more direct marking construction involves an {\it exchangeable sequence} of random variables $\xi_1,\xi_2,\ldots$ and identifies the partition with the classes of the random equivalence relation
$$i\sim j\quad{\rm iff}\quad \xi_i=\xi_j.$$
A natural (though not very common)  possibility is to sample the $\xi_i$'s from the 
{\it unordered frequency distribution}
\begin{eqnarray}\label{nu}
\nu&:=&\sum_{i=1}^\infty F_i\,\delta_{F_i}+ \left(1-\sum_{i=1}^\infty {F_i}\right)\delta_0 \\ \nonumber
&=&\sum_{i=1}^\infty F_i^\downarrow \,\delta_{F_i^\downarrow}+\left(1-\sum_{i=1}^\infty {F_i^\downarrow}\right)\delta_0,
\end{eqnarray}
which is a  discrete random measure
representing the frequencies in a symmetric way.  
The mean measure of $\nu$ yields the distribution of $F_1$ via
$${\mathbb P}[F_1\leq x]={\mathbb E}[\nu[0,x]], ~~x\in[0,1].$$
In Proposition \ref{StaPr} we will show, only using stationarity, that $\nu$ emerges as the limit of its finite analogues for $\Pi_n$.
The $\delta_0$-term is nontrivial in the presence of singleton blocks of  $\Pi$; then $F_k=0$ with positive probability for every $k\geq 1$.
The extreme exchangeable partitions correspond to nonrandom $\nu$, which makes (\ref{nu}) -- viewed as 
a random element of  the space of measures on $[0,1]$ with the topology of  
weak convergence -- an appealing alternative to Kingman's representation via $(F_k^\downarrow)$.

Many interesting  exchangeable partitions of $[n]$ cannot be extended in a consistent way to larger ground sets, and degenerate as $n$ grows.
Thus, the uniform distribution (\ref{Bell}) converges to the singleton partition, as the binary clustering probability vanishes,
$p(2) = B_{n-1}/B_n \sim \log n / n$; also the largest block size is of the order $2\log n$ 
\cite{Sachkov}, rather than grows about linearly in $n$ as occurs for marginals $\Pi_n$ of infinite exchangeable partitions.

Extreme infinite exchangeable partitions can be characterised by the property of {\it dissociation} \cite{Kallenberg}, meaning that restricted partitions on disjoint sets of balls are independent.
This fails for  every nontrivial finite $p_\lambda^*$.

\section{Spreadability}
\noindent
An infinite random partition $\Pi=(\Pi_n)$ is {\it spreadable} if
$
\Pi\stackrel{d}{=}\Pi_{|A}
$
for every infinite subset $A\subset{\mathbb N}$,
where partitions of different sets are identified via the increasing
bijection.
By a theorem of Kallenberg (\cite{Kallenberg}, Corollary 7.41)
\[
\Pi \text{  spreadable}
\quad\Longleftrightarrow\quad
\Pi \text{  exchangeable},
\]
where $\Leftarrow$ is immediate.

Let $D_j$ act on an arbitrary set of balls $A\subset {\mathbb N}$ by deleting ball with the $j$th smallest label  (if
the cardinality of $A$ is at least $j$).
Spreadability is clearly equivalent to invariance under deletion of a single
ball with arbitrary label:
\[
\Pi\stackrel{d}{=}D_j\Pi,
\qquad j\in{\mathbb N}.
\]

In this form the property extends naturally to finite partitions.
We call partition $\Pi_n$ spreadable if
\begin{equation}\label{FinSpr}
D_1\Pi_n \stackrel d= D_2\Pi_n \stackrel d=
\cdots \stackrel d= D_n\Pi_n .
\end{equation}
This property is the appropriate finite counterpart of the infinite spreadability,  because it is hereditary and if it holds   for
every marginal partition of $\Pi=(\Pi_n)$  then $\Pi$ is spreadable.

Write for the distribution
\[
P(\pi_n):={\mathbb P}[\Pi_n=\pi_n],
\qquad \pi_n\in{\mathcal P}_n,
\]
and extend this by the addition rule to partitions of smaller sizes.
Then $\Pi_n$ is spreadable if and only if the following   `spreadability rule' holds:
\begin{equation}\label{SprRule}
\sum_{\pi_n:\,D_i\pi_n=\pi_{n-1}} P(\pi_n)
=
\sum_{\pi_n:\,D_n\pi_n=\pi_{n-1}} P(\pi_n),
\qquad
\pi_{n-1}\in{\mathcal P}_{n-1},\;\;1\le i\le n-1 .
\end{equation}
Thus spreadability for fixed $n$  amounts to a system of linear constraints within  the simplex of
all probability distributions on ${\mathcal P}_n$.

It is immediate that
\[
\Pi_n \text{ exchangeable}
\quad\Longrightarrow\quad
\Pi_n \text{ spreadable}.
\]
The converse fails in general, as the following example shows.
\vskip0.2cm

\begin{example}\label{Ex1}
{\rm 
(A $n=5$ non-exchangeable, spreadable partition.)
Consider a random partition $\Pi_5$ on $\mathcal{P}_5$ with distribution
\[
\begin{minipage}[t]{0.56\textwidth}
\centering
\renewcommand{\arraystretch}{1.15}
\begin{tabular}{l|c}
$\pi_5$ & $P(\pi_5)$ \\
\hline
$12|3|4|5$ & 2 \\
$1|23|4|5$ & 2 \\
$1|2|34|5$ & 2 \\
$1|2|3|45$ & 2 \\
$15|2|3|4$ & 2 \\[1mm]
$13|24|5$ & 1 \\
$13|25|4$ & 1 \\
$14|25|3$ & 1 \\
$14|2|35$ & 1 \\
$1|24|35$ & 1 \\
\hline
scale & 15
\end{tabular}
\end{minipage}
\hfill
\begin{minipage}[t]{0.38\textwidth}
\centering
\renewcommand{\arraystretch}{1.15}
\begin{tabular}{l|c}
$\pi_4$ & $P(\pi_4)$ \\
\hline
$12|3|4$ & 2 \\
$1|23|4$ & 2 \\
$1|2|34$ & 2 \\
$14|2|3$ & 2 \\
$1|24|3$ & 1 \\
$13|2|4$ & 1 \\[1mm]
$13|24$ & 1 \\[1mm]
$1|2|3|4$ & 4 \\
\hline
scale & 15
\end{tabular}
\end{minipage}
\]

This distribution is not exchangeable. For example,
\[
P(12|3|4|5)=\frac{2}{15},
\qquad
P(13|2|4|5)=0.
\]
The spreadability is shown directly, by
 checking that the deletion pushforwards $D_j\Pi_5,$ for $1\leq j\leq5,$
all coincide with the common marginal distribution $\Pi_4$
displayed in the right-hand table. 
}
\end{example}


\vskip0.2cm
The example begs a question of how much the spreadability of finite partitions differs from  exchangeability. Using random marking of the block-tracking sequence 
and the representation through an equivalence relation (Figure~\ref{CanRep}), 
we will employ a connection with random sequences. We call a random sequence $\xi_1,\ldots,\xi_n$ \emph{spreadable} if it satisfies 
\begin{equation}\label{SeqSpr}
D_j (\xi_1, \ldots, \xi_n) \stackrel{d}{=} D_n (\xi_1, \ldots, \xi_n), \qquad j = 1, \ldots, n,
\end{equation}
where $D_j$ acts on the sequence by deleting the coordinate $\xi_j$ and keeping the remaining elements intact.
Equivalently,  spreadability means that all $(n-1)$-dimensional marginal distributions are identical.

\begin{proposition}\label{Prs} 
A partition $\Pi_n$ is spreadable if and only if it can be realised as a partition into classes of a random 
equivalence relation, $i\sim j \Leftrightarrow \xi_i=\xi_j$, where $\xi_1,\ldots,\xi_n$ is a spreadable sequence.
\end{proposition}
\begin{proof}
The direction $\Leftarrow$ is immediate. 
For the converse, let $\eta_1,\ldots,\eta_n$ be an exchangeable sequence with zero probability of ties and independent of $\Pi_n$;
we may choose this to be a random permutation of $[N]$ (for suitable $N$) or an i.i.d. sample from ${\rm U}[0,1]$.
Given a realisation of $\Pi_n$ with block starters $i_1<\cdots<i_k$, we mark these starters with $\eta_1,\ldots,\eta_k$ respectively, and then assign to each element of a block the same mark as its starter. This defines a sequence of marks $\xi_1,\ldots,\xi_n$. The spreadability of this sequence follows from the next general lemma.
\end{proof}

\begin{lemma} \label{repLemma}
For $A\subset[n]$ and $B\subset[n]$, let $\rho: A\to B$ be a bijection. Suppose the random partition $\Pi_n$ has the property that $\Pi|_B$ follows the same distribution as the partition obtained by transporting $\Pi|_A$ to $B$ via $\rho$. Then 
$$(\xi_i, i\in A) \stackrel{d}{=} (\xi_{\rho(i)}, i\in A).$$
\end{lemma}

\begin{proof} 
We illustrate the mechanism by a concrete example. Let $\rho:A\to B$ be given by
$$
\begin{pmatrix}
  2 & 3 & 5 & 6 & 7 & 8 \\
  6 & 3 & 4 & 8 & 10 & 7
\end{pmatrix}.
$$
Conditionally on $\Pi|_A = 258 \mid 37 \mid 6$, the induced marking on $A$ has the form 
\begin{equation}\label{vs}
(\xi_2,\xi_3,\xi_5,\xi_6,\xi_7,\xi_8)=(v_1,v_2,v_1,v_3,v_2,v_1),
\end{equation}
where $v_1, v_2, v_3$ are the marks assigned to the starters of $A$.
The transported partition on $B$ is $3\,10 \mid 467 \mid 8$, which yields the marking sequence:
$$(\xi_3,\xi_4,\xi_6,\xi_7,\xi_{8},\xi_{10})=(w_1,w_2,w_2,w_2,w_3,w_1).$$
Evaluating the sequence on the image points of $\rho$ results in:
\begin{equation}\label{ws}
(\xi_6,\xi_3,\xi_4,\xi_8,\xi_{10},\xi_{7})=(w_2,w_1,w_2,w_3,w_1,w_2).
\end{equation}
The distributions of sequences in (\ref{vs}) and (\ref{ws}) are identical because, by independence and exchangeability of the underlying sequence of marks:
$$
(v_1,v_2,v_3) \stackrel{d}{=} (\eta_1,\eta_2,\eta_3) \quad \text{and} \quad (w_2,w_1,w_3) \stackrel{d}{=} (\eta_1,\eta_2,\eta_3).
$$
\end{proof}
See \cite{Kallenberg} (Lemma 7.39) for a detailed proof of a more general random marking result in the context of infinite symmetries.

A delicate side of the relation between random partitions and sequences is that in the direction `partition $\to$ marking sequence' a special class of sequences of identically distributed random variables 
emerges. We may call a sequence $\xi_1,\ldots,\xi_n$
{\it alphabet-symmetric} if  $(\xi_1,\ldots, \xi_n)\stackrel{d}{=}(\varphi(\xi_1),\ldots, \varphi(\xi_n))$ for every 
 transformation $\varphi$ of the range space which preserves the marginal distribution, so that $\varphi(\xi_1)\stackrel{d}{=}\xi_1$. 
The concept obviously extends to infinite sequences, by requiring the condition to hold for every $n$.

The simplest example of spreadable non-exchangeable sequence is 
the $3$-valued $\xi_1,\xi_2$ with uniform joint distribution over $\{(0,1), (1,2), (2,0)\}$).
Kallenberg (\cite{KallenbergPTRF}, p. 213)  notes that `construction of nontrivial and interesting examples requires both ingenuity and some calculation'.
Two $n=3$ examples presented in \cite{KallenbergPTRF} are three-valued spreadable sequences which, however,  induce simple exchangeable partitions;
and further examples found there (Corollary 2.6) collapse to $n=2$ singleton partitions.
 Example~\ref{Ex1} through  marking the blocks with  uniformly random permutation of $[4]$ yields a more involved non-exchangeable, spreadable (and alphabet-symmetric) sequence $\xi_1,\ldots,\xi_5$
on  a four-letter alphabet.

\begin{proposition}\label{binary}
Every spreadable partition $\Pi_n$ with at most two blocks is exchangeable. 
\end{proposition}
\begin{proof}
We use  marking blocks by permutation of $[2]$
and reduce the claim to sequences, as  in Proposition \ref{Prs}.
 Let $P$ be spreadable distribution on
 binary words $w  \in \{0,1\}^n$. We  show that $P(w)$ depends only on the sum of symbols  $|w|$, which is equivalent to exchangeability.
For a word $w'  \in \{0,1\}^{n-1}$ and   each $j \in \{1, \dots, n\}$, there are  two words $w_{0j}, w_{1j}$ extending  $w'$
by inserting $0$, respectively $1$, in the $j$th position.
Spredability means that for every $w'$  the sum 
$P(w_{0j}) + P(w_{1j})$ does not depend on $j$.
If $w'$ is zero word, then $w_{0j}$ has only zeroes and $|w_{1j}|=1$; therefore $P$ is constant on words with $|w|=1$. By induction, suppose
 $P$ is constant on words with $|w|=k$. Extending $w'$ with $|w'|=k$, we have $|w_{0j}|=k, |w_{1j}|=k+1$.  By the hypothesis $P(w_{0j})$ does not depend on $j$, whence varying 
such $w'$ and $j$ we obtain for  $P(w_{1j})$  the same value. It remains to note that every word $w$ with $|w|=k+1$ can be written as $w = w_{1j}$ for some $w'$ with $|w'|=k$, by deleting a coordinate $j$ such that $w_j = 1$.
\end{proof}

A different argument  for binary sequences is  found in \cite{Kallenberg} (Theorem 1.13 (iii)).  
The exchangeability of marginal partition $\Pi_4$ in Example~\ref{Ex1} suggested the next result, which has no analogue for sequences of length $n=2,3$.

\begin{proposition}\label{n=4}
Every spreadable random partition $\Pi_4$ of $[4]$ is exchangeable.
\end{proposition}
\begin{proof}
Write $P(\pi):=\mathbb P[\Pi_4=\pi]$ for $\pi\in\mathcal P_4$. We verify
that the system of linear equations defined by \eqref{SprRule} implies that
$P$ is constant on each orbit ${\mathcal O}_\lambda$
of the symmetric group $\mathfrak S_4$.

{\bf 1. $\lambda=(3,1)$.} Consider $\pi_3=123\in\mathcal P_3$. The spreadability
condition \eqref{SprRule} requires that the total  mass projecting from $\mathcal P_4$
to $\pi_3$ is the same for each $D_i$. The partition $1234$ 
maps to $123$ under every $D_i$, contributing an identical $P(1234)$ term which cancels 
out. Looking at the remaining preimages of shape $(3,1)$, 
under $D_4$ the unique preimage is $123|4$. Under $D_3$, ball $3$
is deleted and ball $4$ is relabelled as $3$, so the unique preimage is
$124|3$. Likewise, under $D_2$ and $D_1$, the unique preimages are $134|2$
and $234|1$, respectively. Therefore,  \eqref{SprRule} gives
\begin{equation}\label{eq:shape31}
P(123|4)
=
P(124|3)
=
P(134|2)
=
P(1|234).
\end{equation}

{\bf 2.  $\lambda=(2,1,1)$.}
Consider  $\pi_3=1|2|3$. Applying
\eqref{SprRule} and comparing the four deletion operators yields a linear
system for the six partitions of shape $(2,1,1)$. From this we find 
\[
P(12|3|4)
=
P(13|2|4)
=
P(14|2|3)
=
P(1|23|4)
=
P(1|24|3)
=
P(1|2|34).
\]

{\bf 3. $\lambda=(2,2)$.}
Now consider the target partition $\pi_3=12|3\in\mathcal P_3$. Under $D_4$,
the preimages are $124|3$, $12|34$ and $12|3|4$. Under $D_2$, the preimages are
$123|4$, $13|24$ and $13|2|4$. Thus, using \eqref{eq:shape31} and the equalities above, \eqref{SprRule} gives
$P(12|34)=P(13|24).$
Repeating the same argument with $D_1$ we
obtain
\[
P(12|34)
=
P(13|24)
=
P(14|23).
\]

{\bf 4. $\lambda=(1,1,1,1)$ and $\lambda=(4)$.} These are the  ranked shapes corresponding to single-element orbits.

We see that spreadable $P$ must be constant on every orbit, hence $\Pi_4$ is exchangeable.
\end{proof}

Geometrically,  spreadable distributions on ${\mathcal P}_n$ comprise a polytope, hence it is natural to look for its extreme points. 
The exchangeable distributions comprise a distinguished simplex spanned on partitions with uniform distribution (possessing the probability function $p^*_\lambda$) supported on a single orbit $\mathcal O_\lambda$;
these are the most obvious  candidates to be scrutinised.  

For sequences $\xi_1,\ldots,\xi_n$ with nonrandom empirical measure $\beta =\sum_{j=1}^n\delta_{\xi_j}$ the spreadability implies exchangeability (see \cite{Kallenberg} Lemma 2.2);
but neither the result nor its proof apply to spreadable partitions, since the alphabet-symmetric marking sequence has nonrandom $\beta$, and conditioning on $\beta$ in general destroys spreadability
(otherwise, by Kallenberg's result, every spreadable sequence were exchangeable).
To show the counterpart for partitions, we use a simple observation. First note that for $q\in{\mathcal P}_{n-1}$ the fibre   $D_j^{-1}q$ comprises  partitions of $[n] $ obtained by shifting labels
$j,\ldots,n-1\mapsto j+1,\ldots,n$ then allocating ball $j$   in one of the boxes present in $q$ or opening a new box.

\begin{lemma}\label{nostep}
Let $\pi_n\in {\mathcal P}_n$ be a partition with ranked shape $\lambda=(\lambda_1,\ldots,\lambda_k)$  satisfying $\lambda_i>\lambda_{i+1}+1$  {\rm (}respectively, $\lambda_i=1${\rm )} for some $i\leq k$,
where $\lambda_{k+1}:=0$.
If  $q=D_j\pi_n$ for some $j$ has ranked  shape $\lambda'$ with reduced part $\lambda_i'=\lambda_i-1$ {\rm (}respectively, lesser number of singletons{\rm)}, then
$D^{-1}_\ell q \cap\mathcal O_\lambda$ is a one-element set 
 for every $\ell\in [n]$. 
\end{lemma}
\begin{proof} The example best explains the matters.
Suppose by deleting a ball $a$ the  ranked shape $(5,5,3,3,1,1)$ was reduced to $(5,4,3,3,1,1)$, then there is no ambiguity where a ball $b$ needs to be appended to return to $(5,5,3,3,1,1)$. 
If $a$ was a singleton, then $b$ must be a singleton.
Compare with shape $(5,4,3,3,1,1)$ which after reducing to $(5,3,3,3,1,1)$ can be recovered in a threefold way.
\end{proof}

\begin{proposition}\label{SingleOrbit}
 If $\Pi_n$ is a spreadable partition with constant ranked shape,
$s^\downarrow(\Pi_n)=\lambda$ a.s.,
 then $\Pi_n$ is  exchangeable, with symmetric probability function $p^*_\lambda$, and is extreme among all
spreadable partitions.
\end{proposition}
\begin{proof} We argue in terms of distribution $P$ on ${\mathcal P}_n$ using induction in $n$. The smallest part  $\lambda_k$ of $\lambda=(\lambda_1,\ldots,\lambda_k)$ always satisfies conditions of Lemma \ref{nostep}. 
Among the boxes of size  $\lambda_k$ we tag the one with smallest starter. To be definite, let $\lambda_k=3$, and for 
$$0\leq a+1\leq b+1\leq c+1\leq n+1$$
let $A(a,b,c)$ be the set of partitions $\pi_n\in {\mathcal O}_\lambda$ whose tagged box content is $abc$.

Suppose $0<a+1$.  For $\pi_n\in A(a,b,c)$ the partition $q=D_a \pi_n$ satisfies   $D_a^{-1}q=\{\pi_n\}$ as well as   $D_{a-1}^{-1}q=\{\pi_n'\}$ for  the partition $\pi_n'\in A(a-1,b,c)$ satisfying 
$\pi_n'= \sigma \pi_n$, where $\sigma$ is the adjacent transposition $(a-1, a)$.
By the uniqueness, the spreadability rule (\ref{SprRule}) in the form

\begin{equation}\nonumber
\sum_{\pi_n\in{\mathcal O}_\lambda :\,D_a\pi_n=q} P(\pi_n)
=
\sum_{\pi_n\in{\mathcal O}_\lambda:\,D_{a-1}\pi_n=q} P(\pi_n),
\end{equation}
 yields $P(\pi_n)=P(\pi_n')$.
It follows that $\sigma$ acts as a $P$-preserving bijection between $A(a,b,c)$ and $A(a-1,b,c)$.  This way exchanging the balls in the tagged box in the direction of smaller labels,
we obtain such bijections between each $A(a,b,c)$ and $A(1,2,3)$. Now, conditionally on $A(1,2,3)$ the distribution $P$ on partitions of $[n]\setminus[3]$ is spreadable and supported
on the single orbit ${\mathcal O}_{(\lambda_1,\ldots,\lambda_{k-1})}$, therefore by the induction hypothesis the conditional distribution is exchangeable.
Since all strata $A(a,b,c)$ have the same probability, $P$ is the uniform distribution on ${\mathcal O}_\lambda$, with probability function $p^*_\lambda$.

The case $\lambda_k=1$ is special in that the tagged singleton box $\{a\}$ is deleted in $q=D_a\pi_n$, but gets resurrected by the starter $a-1$   in $\pi_n'$. 

For the  claim on extremality,
the distribution $P$ can be decomposed only as a mixture of spreadable distributions supported on ${\mathcal O}_\lambda$, hence by the above all components coincide with $P$.
\end{proof}

Stepping aside from our main line,  as a by-product we obtain a result on the alphabet-symmetric spreadable sequences. 

\begin{corollary}\label{Cor1}
Let $\xi_1,\ldots,\xi_n$ be a spreadable sequence with values in $[k]$. If the random measure $\sum_{j=1}^n \delta_{\xi_j}$ is invariant and ergodic under permutations $\tau\in{\mathfrak S}_k$ acting
as $\tau(\xi_1),\ldots,\tau(\xi_n)$, then $\xi_1,\ldots,\xi_n$ is exchangeable.
\end{corollary}
\begin{proof}
By the assumption,
the partition $\Pi_n$ 
induced by the sequence via $i\sim j\Leftrightarrow\xi_i=\xi_j$
is spreadable and single-orbit, therefore $\Pi_n$ is exchangeable according to Proposition \ref{SingleOrbit}.
Relative to this $\Pi_n$,    by Lemma \ref{repLemma} the sequence $\xi_1,\ldots,\xi_n$ is obtained  by marking the block-tracking word with uniformly random permutation $\tau(1),\ldots  \tau(k)$.
\end{proof}


We see that, unlike partial exchangeability, the genuine finite spreadability is not reduced to a refinement of orbits, but rather allocates mass  to parts of different orbits of the symmetric group. The following example of extreme spreadable partition demonstrates a phenomenon of uniform mass transition between the levels.%

\begin{example}\label{Ex2}{\rm (A $n=5$ extreme spreadable partition with a uniform distribution.)
Consider the random partition $\Pi_5$ on $[5]$ with uniform distribution on a set intersecting  three  ${\mathfrak S}_5$-orbits in $\mathcal P_5$:
\begin{equation*}
\begin{aligned}
&14 \mid 235, \quad 135 \mid 24, \quad 125 \mid 34, \quad 123 \mid 45, \\
&1 \mid 2 \mid 345, \quad 1 \mid 245 \mid 3, \quad 1 \mid 234 \mid 5, \quad 145 \mid 2 \mid 3, \quad 134 \mid 2 \mid 5, \quad 124 \mid 3 \mid 5, \\
&12 \mid 35 \mid 4, \quad 13 \mid 25 \mid 4, \quad 15 \mid 23 \mid 4.
\end{aligned}
\end{equation*}

Notably, all five deletions act as bijections, resulting in the uniform distribution on a union of four orbits in $\mathcal P_4$:

\begin{minipage}[t]{0.48\textwidth}
\centering
\begin{tabular}{l | l}
$\pi_5$ & $D_1\pi_5$ \\
\hline
$14 \mid 235$ & $124 \mid 3$ \\
$135 \mid 24$ & $13 \mid 24$ \\
$125 \mid 34$ & $14 \mid 23$ \\
$123 \mid 45$ & $12 \mid 34$ \\
$1 \mid 2 \mid 345$ & $1 \mid 234$ \\
$1 \mid 245 \mid 3$ & $134 \mid 2$ \\
$1 \mid 234 \mid 5$ & $123 \mid 4$ \\
$145 \mid 2 \mid 3$ & $1 \mid 2 \mid 34$ \\
$134 \mid 2 \mid 5$ & $1 \mid 23 \mid 4$ \\
$124 \mid 3 \mid 5$ & $13 \mid 2 \mid 4$ \\
$12 \mid 35 \mid 4$ & $1 \mid 24 \mid 3$ \\
$13 \mid 25 \mid 4$ & $14 \mid 2 \mid 3$ \\
$15 \mid 23 \mid 4$ & $1 \mid 23 \mid 4$
\end{tabular}
\end{minipage}
\hfill
\begin{minipage}[t]{0.48\textwidth}
\centering
\begin{tabular}{l | l}
$\pi_5$ & $D_2\pi_5$ \\
\hline
$14 \mid 235$ & $13 \mid 24$ \\
$135 \mid 24$ & $124 \mid 3$ \\
$125 \mid 34$ & $14 \mid 23$ \\
$123 \mid 45$ & $12 \mid 34$ \\
$1 \mid 2 \mid 345$ & $1 \mid 234$ \\
$1 \mid 245 \mid 3$ & $1 \mid 2 \mid 34$ \\
$1 \mid 234 \mid 5$ & $1 \mid 23 \mid 4$ \\
$145 \mid 2 \mid 3$ & $134 \mid 2$ \\
$134 \mid 2 \mid 5$ & $123 \mid 4$ \\
$124 \mid 3 \mid 5$ & $13 \mid 2 \mid 4$ \\
$12 \mid 35 \mid 4$ & $1 \mid 24 \mid 3$ \\
$13 \mid 25 \mid 4$ & $1 \mid 23 \mid 4$ \\
$15 \mid 23 \mid 4$ & $14 \mid 2 \mid 3$
\end{tabular}
\end{minipage}

\vspace{0.6cm}
\begin{minipage}[t]{0.31\textwidth}
\centering
\begin{tabular}{l | l}
$\pi_5$ & $D_3\pi_5$ \\
\hline
$14 \mid 235$ & $13 \mid 24$ \\
$135 \mid 24$ & $14 \mid 23$ \\
$125 \mid 34$ & $124 \mid 3$ \\
$123 \mid 45$ & $12 \mid 34$ \\
$1 \mid 2 \mid 345$ & $1 \mid 2 \mid 34$ \\
$1 \mid 245 \mid 3$ & $1 \mid 234$ \\
$1 \mid 234 \mid 5$ & $1 \mid 23 \mid 4$ \\
$145 \mid 2 \mid 3$ & $134 \mid 2$ \\
$134 \mid 2 \mid 5$ & $13 \mid 2 \mid 4$ \\
$124 \mid 3 \mid 5$ & $123 \mid 4$ \\
$12 \mid 35 \mid 4$ & $1 \mid 23 \mid 4$ \\
$13 \mid 25 \mid 4$ & $1 \mid 24 \mid 3$ \\
$15 \mid 23 \mid 4$ & $14 \mid 2 \mid 3$
\end{tabular}
\end{minipage}
\hfill
\begin{minipage}[t]{0.31\textwidth}
\centering
\begin{tabular}{l | l}
$\pi_5$ & $D_4\pi_5$ \\
\hline
$14 \mid 235$ & $1 \mid 234$ \\
$135 \mid 24$ & $134 \mid 2$ \\
$125 \mid 34$ & $124 \mid 3$ \\
$123 \mid 45$ & $123 \mid 4$ \\
$1 \mid 2 \mid 345$ & $1 \mid 2 \mid 34$ \\
$1 \mid 245 \mid 3$ & $1 \mid 24 \mid 3$ \\
$1 \mid 234 \mid 5$ & $1 \mid 23 \mid 4$ \\
$145 \mid 2 \mid 3$ & $14 \mid 2 \mid 3$ \\
$134 \mid 2 \mid 5$ & $13 \mid 2 \mid 4$ \\
$124 \mid 3 \mid 5$ & $1 \mid 24 \mid 3$ \\
$12 \mid 35 \mid 4$ & $12 \mid 34$ \\
$13 \mid 25 \mid 4$ & $13 \mid 24$ \\
$15 \mid 23 \mid 4$ & $14 \mid 23$
\end{tabular}
\end{minipage}
\hfill
\begin{minipage}[t]{0.31\textwidth}
\centering
\begin{tabular}{l | l}
$\pi_5$ & $D_5\pi_5$ \\
\hline
$14 \mid 235$ & $14 \mid 23$ \\
$135 \mid 24$ & $13 \mid 24$ \\
$125 \mid 34$ & $12 \mid 34$ \\
$123 \mid 45$ & $123 \mid 4$ \\
$1 \mid 2 \mid 345$ & $1 \mid 2 \mid 34$ \\
$1 \mid 245 \mid 3$ & $1 \mid 24 \mid 3$ \\
$1 \mid 234 \mid 5$ & $1 \mid 234$ \\
$145 \mid 2 \mid 3$ & $14 \mid 2 \mid 3$ \\
$134 \mid 2 \mid 5$ & $124 \mid 3$ \\
$124 \mid 3 \mid 5$ & $124 \mid 3$ \\
$12 \mid 35 \mid 4$ & $12 \mid 3 \mid 4$ \\
$13 \mid 25 \mid 4$ & $13 \mid 2 \mid 4$ \\
$15 \mid 23 \mid 4$ & $1 \mid 23 \mid 4$
\end{tabular}
\end{minipage}
\vskip0.2cm

The distribution of $\Pi_5$ 
is invariant under the cyclic subgroup of order four,  generated by the cycle $\sigma = (1 \; 2 \; 3 \; 5)$.
Realising the partition via the marking construction, we obtain a genuinely spreadable $[3]$-valued
random sequence $\xi_1, \ldots, \xi_5$, by choosing a block-tracking word uniformly from the list
\vskip0.1cm
\begin{equation*}
\begin{aligned}
&12212, \quad 12121, \quad 11221, \quad 11122, \\
&12333, \quad 12322, \quad 12223, \\
&12311, \quad 12113, \quad 11213, \\
&11232, \quad 12132, \quad 12231,
\end{aligned}
\end{equation*}
and replacing its letters $1,2,3$ with $\tau(1),\tau(2),\tau(3)$, a uniform permutation of $[3]$.
}
\end{example}
\vskip0.2cm

The existence of non-exchangeable partitions among spreadable implies that the polytope of spreadable partitions is not a simplex.
Indeed, let $Q$ be non-exchangeable spreadable and $P$ fully suported exchangeable, for instance (\ref{Bell}). Then  for sufficiently small
positive $\alpha$ the identity
$$ P=\frac{1+\alpha}{2}  \,\frac{P+\alpha Q}{1+\alpha}+ \frac{1+\alpha}{2}\,  \frac{P-\alpha Q}{1-\alpha}$$
shows that $P$ is the middlepoint between two non-exchangeable spreadable distributions. Indeed, positivity of the components and normalisation are implied by the linearity,
and both are nonexchangeable due to the asymmetric $Q$. On the other hand, $P$ is decomposable over exchangeable distributions, hence its mixture  representation is not unique
and the polytope is not a simplex. 
This also implies that the simplex of exchangeable distributions is not an exposed face lying on a supporting hyperlan, rather belongs to a section of the polytope.
We give a concrete $n=5$ example, though use deformation of $P$ in the direction of a signed measure.

\begin{example}\label{Ex3}{\rm
(A mixture of genuinely spreadable partitions may be exchangeable.)
 We construct a {\it signed} measure $S$ on ${\mathcal P}_5$ with zero total mass:
\[
\begin{array}{l|r}
\pi_5 & S(\pi_5) \\
\hline
1 \mid 2 \mid 3 \mid 45 & -1 \\
1 \mid 2 \mid 34 \mid 5 & -1 \\
1 \mid 2 \mid 35 \mid 4 & -1 \\
1 \mid 2 \mid 345      & +1 \\[1mm]
12 \mid 3 \mid 4 \mid 5 & -2 \\[1mm]
12 \mid 3 \mid 45      & +1 \\
12 \mid 34 \mid 5      & +1 \\
12 \mid 35 \mid 4      & +1 \\
12 \mid 345            & -1 \\
\hline
\textbf{Total} & 0
\end{array}
\]
This has the property that $D_j S$ is the measure with a sole weight $-2$ on the singleton partition $1|2|3|4$,
for $1\leq j\leq 5$. Let $P$ be a fully supported exchangeable distribution on $\mathcal P_5$, for instance (\ref{Bell}).
Due to these properties and since $S$ is not exchangeable, for small enough $\alpha$ the sum $P\pm\alpha S$ is a spreadable distribution on $\mathcal P_5$. Therefore $P$ is a barycenter of two spreadable nonexchangeable distributions.
}
\end{example}
\vskip0.0cm
For $n=5$ the polytope of spreadable distributions has affine dimension $17$.
An automatic search returns hundreds of extreme points, which implies that the
decomposition over  extremes  is by far not unique.
The nonsimplicial geometry of the   polytope  means 
that it is impossible to represent in a unique way the generic distribution as a mixture of some basic structures;
in this sense spreadability does not admit a sufficient statistic. 
In  the extreme spreadable partition of Example \ref{Ex2}  even the number of blocks varies across the suport of the distribution.

Introducing finite spreadable sequences Kallenberg \cite{KallenbergPTRF} wrote 
`we are immediately struck by the peculiar lack of symmetry or simple pattern.
It is indeed remarkable that so many wondrous properties are hidden behind such
an apparent complexity and disorder. The present irregularity is in sharp contrast
with the trite symmetry in the exchangeable case'. 
This view is confirmed by the spreadable partitions: the  monstrous combinatorial object 
collapses asymptotically  to the highly homogeneous simplicial class of exchangeable partitions.

\section{Stationarity}

\noindent
An infinite random partition $\Pi=(\Pi_n)$ is {\it stationary} if it  is invariant under the unit shift $i\mapsto i+1$, that is
\begin{equation}\label{stat}
\Pi\stackrel{d}{=} \Pi|_{{\mathbb N}\setminus\{1\}}.
\end{equation}
We can express (\ref{stat}) as the elementary deletion invariance
\begin{equation}\label{delstat}
D_1\Pi\stackrel{d}{=} \Pi,
\end{equation}
where $D_1$ acts by deleting ball $1$ then shifting down all other labels.
 Two further equivalent conditions are formulated in terms of the bivariate array $(R_{ij})$, where
$R_{ij}$ denotes the indicator of the event that $i,j$ belong to the same block of $\Pi$:
\begin{enumerate}
\item[(i)]  $(R_{ij})\stackrel{d}{=} (R_{i+1,j+1})$,
\item[(ii)]  
$R_{ij}=1(\xi_i=\xi_j),$
\end{enumerate}
for some stationary  sequence of random variables $\xi_1,\xi_2,\ldots$. The equivalence of (ii) with 
(\ref{stat}) follows by the marking device (similarly to Proposition \ref{Prs}) using an alphabet-symmetric sequence.

The existence of block frequencies and their invariance under size-biased  permutation are the cornerstones of the theory of exchangeable partitions. We show next that 
the stationarity alone implies the existence of frequencies and that the  first frequency is a size-biased pick from all.

\begin{proposition}\label{StaPr}
For a stationary partition $\Pi$, the blocks arranged by increase of their starters possess asymptotic frequencies $F_1,F_2,\ldots$. Conditionally on the unordered frequency distribution
{\rm (\ref{nu})}, the frequency of the first block has distribution $\nu$.
\end{proposition}

\begin{proof}
For each $n$, let $n^{-1}s(\Pi_n)=(F_{1,n},F_{2,n},\ldots)$ be the normalised shape vector (padded by zeroes), and define the  probability measure
\[
\nu_n: = \sum_i F_{i,n}\,\delta_{F_{i,n}}.
\]
For $m\ge1$, writing down the moments of $\nu_n$ and expressing in two ways the  proportion of ordered $m$-tuples of balls lying in the same block of $\Pi_n$ we obtain
\begin{equation}\label{MomCon}
\int_0^1 x^{m-1}\,\nu_n({\rm d}x)=\sum_i F_{i,n}^{\,m}
=
\frac{1}{n^{m}}
\sum_{(i_1,\ldots,i_m)\in[n]^m}
R_{i_1,i_2}\cdots R_{i_{m-1},i_m},
\end{equation}
By stationarity and the Ergodic Theorem applied to the shift-invariant array $(R_{ij})$, these averages converge almost surely, and therefore 
the Hausdorff Moment Theorem implies that 
$\nu_n \Rightarrow \nu$ for some random probability measure $\nu$ on $[0,1]$. This yields existence of block frequencies $F_i=\lim_{n\to\infty} F_{i,n}$.

For the second part, note that
\[
F_1 = \lim_{n\to\infty} \frac{1}{n}\sum_{i=1}^n R_{i1}.
\]
Hence for $m\ge1$,
\[
F_1^{m-1}
=
\lim_{n\to\infty}
\frac{1}{n^{m-1}}
\sum_{(i_1,\ldots,i_{m-1})\in[n]^{m-1}}
R_{i_1,1}\cdots R_{i_{m-1},1}.
\]
For fixed $j$, split this sum into tuples with all indices in $\{j,\ldots,n\}$ and the remainder, and note that the latter contribute a negligible proportion. 
On $\{j,\ldots,n\}^{m-1}$, shifting indices by $j-1$ yields, by stationarity,
$$(R_{i_1,j},\ldots,R_{i_{m-1},j})
\;\stackrel{d}{=}\;
(R_{i_1-j+1,1},\ldots,R_{i_{m-1}-j+1,1}),
$$
so the resulting sum has the same asymptotic behaviour as for $j=1$. Hence the averages
\[
A_n(j):=\frac{1}{n^{m-1}}\sum_{(i_1,\ldots,i_{m-1})} R_{i_1,j}\cdots R_{i_{m-1},j}
\]
all have the same limit. Now also averaging over $j$ yields
\[
\frac{1}{n^{m}}
\sum_{(j,i_1,\ldots,i_{m-1})\in[n]^m}
R_{i_1,j}\cdots R_{i_{m-1},j}
\;\to\;
F_1^{m-1}.
\]
Recalling (\ref{MomCon})
and passing to the limit gives
\[
\mathbb{E}[F_1^{m-1}\mid \nu]
=
\int_0^1 x^{m-1}\,\nu(dx),
\]
which identifies the conditional law of $F_1$ as $\nu$.
\end{proof}

In general,    the frequencies do not determine the distribution of stationary $\Pi$, nor is 
$F_2$  the second size-biased pick from the multiset of frequencies.

\vskip0.2cm
We turn next to finite partitions. We call a partition $\Pi_n$ of $[n]$ stationary if it satisfies
\begin{equation}\label{FinSt}
D_1\Pi_n \stackrel d= D_n\Pi_n.
\end{equation}
The rationale behind this definition is that (\ref{stat}) holds if and only if each marginal distribution satisfies the corresponding condition. To support this concept of finite stationarity, note that (\ref{FinSt}) is hereditary: if $\Pi_{n-1} := D_n \Pi_n$, then
\[
D_1 \Pi_{n-1} \stackrel d= D_{n-1} \Pi_{n-1}.
\]

For sequences, cyclicity (or periodicity) is sometimes viewed as a finite analogue of stationarity, but this property is not hereditary and can only be used to construct projective limits along subsequences. Nevertheless, for fixed $n$, cyclicity is a distinguished  symmetry property.
Let $\sigma_n^*:=(1~2~\cdots~n)$ be the cyclic shift acting on $[n]$. We call $\Pi_n$ {\it cyclically invariant} if
\[
\sigma_n^*\Pi_n \stackrel{d}{=} \Pi_n.
\]
Denoting by ${\mathfrak C}_n$ the cyclic group generated by $\sigma_n^*$, cyclic invariance means that $\Pi_n$ is conditionally uniform on every ${\mathfrak C}_n$-orbit.
One readily checks that $D_1 \circ \sigma_n^* = D_n$, hence a cyclically invariant $\Pi_n$ is necessarily stationary.

\begin{proposition}
For cyclically invariant  $\Pi_n$, the size of the first block is a size-biased pick from all block sizes.
\end{proposition}
\begin{proof}  
Let $\pi_n$ be some value of  $\Pi_n$ with $s^{\downarrow}(\pi_n)=(\lambda_1,\ldots,\lambda_k)$. Label the boxes accordingly as $b_1,\ldots,b_k$. 
Shifting $\pi_n$  cyclically by a permutation chosen uniformly from  ${\mathfrak C}_n$ preserves the probability of partition and maps ball $j$ to ball $1$ with probability $1/n$, hence each box $b_i$ will 
contain ball $1$ with probability $\lambda_i/n$. 
\end{proof}
A weird example of stationary partition, $\Pi_n=1n|2|\cdots|n-1$ a.s., shows that the size-biased pick property fails in general for $n>2$. Compare with infinite partitions in Proposition \ref{StaPr}.

Clearly, the stationarity is much weaker than the spreadability condition (\ref{FinSpr}).
The polytope of all distributions on $\mathcal P_n$ has dimension $B_n-1$, and the 
 linear constraints imposed by stationarity remove  only a small fraction of the dimension.

\begin{proposition}
The polytope of stationary distributions on $\mathcal P_n$ has dimension $B_n - B_{n-1}$.
\end{proposition}

\begin{proof}
Stationarity (\ref{FinSt}) means that  distribution $P$  satisfies the
linear system
\begin{equation}\label{StaRule}
\sum_{\rho\in \mathcal P_n:\,D_n \rho = q} P(\rho)
-
\sum_{\rho\in \mathcal P_n:\,D_1 \rho = q} P(\rho)
=0, 
\qquad q\in \mathcal P_{n-1},
\end{equation}
together with the normalisation $\sum_{\rho}P(\rho)=1$. Thus we have $B_n$ variables subject to $B_{n-1}$ linear constraints.

We may
interpret the coefficient matrix of \eqref{StaRule} as the vertex--edge incidence matrix of a directed graph on $\mathcal P_{n-1}$: for each $\rho\in\mathcal P_n$ introduce a directed edge from $q_1=D_n \rho$ to 
$q_2=D_1 \rho$. In other words, partition $q_2$ is obtained from $q_1$ by adding ball $n$ and then deleting ball $1$.

We claim that this directed graph is strongly connected. Indeed, from any $q\in\mathcal P_{n-1}$ there is a path to the singleton partition by repeatedly applying the operation `attach $n$ as a singleton
then remove ball $1$'. Conversely, starting from the singleton partition, one can reach any $q$: for $k=1,\dots,n-1$, place ball $n$ where ball $n-k$ belongs in $q$. Removing loops, this gives a path of length at most $n-1$. 
To illustrate, consider $q=14|23$. Starting from the singleton partition we reach $q$ in two steps:
\[
1|2|3|4 \to 1|2|34 \to 14|23.
\]
More explicitly, one may track the intermediate relabellings as
\[
{\rm xx}|{\rm xx}\to 4{\rm x}|{\rm xx}\to 3{\rm x}|4{\rm x}\to 2{\rm x}|34\to 14|23.
\]
In general, the number of moves required equals $n-1$ minus the number of nonsingleton blocks of $q$.

From the strong   connectedness, by the standard result on incidence matrices of directed graphs (\cite{Biggs}, Proposition 4.3), the rank of \eqref{StaRule} is $|\mathcal P_{n-1}|-1 = B_{n-1}-1$. Adding the normalisation condition yields total rank $B_{n-1}$.
Hence the dimension of the solution space is
$
B_n - B_{n-1},
$
as claimed.
\end{proof}

Exchangeable single-orbit distributions $p_\lambda^*$ are not extreme points of the polytope of stationary distributions on ${\mathcal R}_n$  (except for two trivial cases), since such orbits can be decomposed into smaller cyclic orbits. 
In particular, the simplest ${\mathfrak S}_n$-orbit with ranked shape $(2,1,1,\ldots,1)$ splits into $\lfloor n/2 \rfloor$ cyclic orbits of length $n$, and into $\lfloor n/2 \rfloor - 1$ cyclic orbits of length $n$ together with one of length $n/2$.
Surprisingly, even distributions supported on a single cyclic orbit need not be extreme stationary, which makes major difference with the spreadable partitions (cf.\ Proposition~\ref{SingleOrbit}).

\begin{example}\label{E3}   
{\rm (Uniform distribution on a cyclic orbit is not extreme among stationary.) 
The set of stationary distributions supported on the $\mathfrak{C}_4$-cyclic orbit of the partition\\ $1 \mid 2 \mid 34$ is an interval with endpoints $P_1$ and $P_2$, 
as in the table,
\begin{longtable}{l !{\vrule width 1.5pt} l !{\vrule width 1.5pt} l !{\vrule width 1.5pt} l !{\vrule width 1.5pt} l}
\multicolumn{1}{c!{\vrule width 1.5pt}}{$\pi_4$} & 
\multicolumn{1}{c!{\vrule width 1.5pt}}{$D_1\pi_4$} & 
\multicolumn{1}{c!{\vrule width 1.5pt}}{$D_4\pi_4$} & 
\multicolumn{1}{c!{\vrule width 1.5pt}}{$P_1$} & 
\multicolumn{1}{c}{$P_2$} \\
\hline
\endhead
$1 \mid 2 \mid 34$ & $1 \mid 23$       & $1 \mid 2 \mid 3$ & $1/3$ & $0$ \\
$14 \mid 2 \mid 3$ & $1 \mid 2 \mid 3$ & $1 \mid 2 \mid 3$ & $0$   & $1$ \\
$12 \mid 3 \mid 4$ & $1 \mid 2 \mid 3$ & $12 \mid 3$       & $1/3$ & $0$ \\
$1 \mid 23 \mid 4$ & $12 \mid 3$       & $1 \mid 23$       & $1/3$ & $0$ \\
\end{longtable}
\noindent
whose mixture 
\[
U = \frac{3}{4}P_1 + \frac{1}{4}P_2
\]
is the uniform distribution on the cyclic orbit.
}
\end{example}

\vskip0.2cm
Note that in Example~\ref{E3}  the Dirac measure on $1n\, |\,2 \,| \,3 \,|\cdots |\, n-1$ is involved.
A minor  generalisation is the following  phenomenon of propagation of stationarity to a higher level:

\begin{proposition}
Let $\Pi_n$ be a random partition of $[n]$ such that elements $1$ and $n$ belong to the same block almost surely, and let
$\Pi_{n-1} := D_n\Pi_n$ be cyclically invariant. 
Then $\Pi_n$ is stationary.
\end{proposition}

\begin{proof}
A marking sequence for such $\Pi_n$ has the form $\xi_1, \ldots, \xi_{n-1}, \xi_1$.
By cyclic invariance of the subsequence of length $n-1$,
\[
(\xi_1, \ldots, \xi_{n-1}) \stackrel{d}{=} (\xi_2, \ldots, \xi_{n-1}, \xi_1).
\]
The left-hand side corresponds to deleting the first element from $(\xi_1, \ldots, \xi_{n-1}, \xi_1)$, while the right-hand side corresponds to deleting the last, which establishes stationarity.
\end{proof}

\section{Partial exchangeability + stationarity}

\noindent
Pitman and Yakubovich \cite{PY} proved that 
a partially exchangeable and stationary (PES) infinite partition $\Pi=(\Pi_n)$ must be exchangeable.
By partial exchangeability 
the distribution of PES-partition $\Pi_n$ is determined by a probability function $p$. The stationarity condition $D_1\Pi_n\stackrel{d}{=} D_n\Pi_n$  
becomes a linear system of equations
\begin{eqnarray}\nonumber
p(1, \lambda_1, \dots, \lambda_k) + \sum_{j=1}^{k} p(\lambda_j + 1, \lambda_1, \dots, \lambda_{j-1}, \lambda_{j+1}, \dots, \lambda_k) =\\   \label{PES}
 \sum_{i=1}^{k} p(\lambda_1, \dots, \lambda_i + 1, \dots, \lambda_k)+ p(\lambda_1, \dots, \lambda_k, 1).
\end{eqnarray}
for every composition $(\lambda_1,\ldots,\lambda_k)$ of $n-1$. The right-hand side expansion of 
$p(\lambda_1, \dots, \lambda_k)$ 
in this formula comes from  the consistency condition  (\ref{p-Rec}). The left-hand side is the expansion over the fibre
of $D_1$: if $D_1\pi_n=\rho$, then 
$\pi_n$ can be obtained by shifting  ball labels in $\rho$, inserting ball $1$ in the partition, and moving the block receiving that ball to the front of the sequence of blocks.
For instance, if  $s(\rho)=(a,b,c)$ for some $\rho\in{\mathcal P}_{n-1}$, the corresponding equation specialises as
\begin{eqnarray*}
p(1,a,b,c)+p(a+1,b,c)+p(b+1,a,c)+p(c+1,a,b)=\\p(a+1,b,c)+p(a,b+1,c)+p(a,b,c+1)+p(a,b,c,1).
\end{eqnarray*}
Algebraically, the cited result 
means that if (\ref{PES})  holds for every $n$ then $p$ is a symmetric function for every shape length $k$.

We will explore the connection to exchangeability  for finite PES-partitions.
In some aspects the situation is similar to spreadable partitions (Propositions \ref{binary}, \ref{n=4}, \ref{SingleOrbit}).

\begin{proposition} For PES partition $\Pi_n$ any of the following two conditions implies  exchangeability:
\begin{itemize}
\item[\rm (i)] the partition has at most two blocks almost surely,
\item[\rm (ii)] $n\leq 4$.
\end{itemize}
\end{proposition}

\begin{proof}
(i) The expansion of $p(a,b)$ gives after cancellation $p(b+1,a)=p(a,b+1)$.

\noindent
(ii) Symmetry of $p$  follows by straightforward algebra from the expansions for four compositions of $3$,  
\begin{eqnarray*}
p(3)  = &p(4) + p(1,3) = p(4) + p(3,1) \\
p(2,1)=& p(1,2,1) + p(3,1) + p(2,2) = p(3,1) + p(2,2) + p(2,1,1) \\
p(1,2) = &p(1,1,2) + p(2,2) + p(3,1) = p(2,2) + p(1,3) + p(1,2,1) \\
p(1,1,1)= &p(1,1,1,1) + 3 p(2,1,1)  = p(2,1,1) + p(1,2,1) + p(1,1,2) + p(1,1,1,1)
\end{eqnarray*}
\end{proof}

\begin{proposition}
Every PES-partition $\Pi_n$ of $[n]$ with constant ranked shape, $s^\downarrow(\Pi_n)=\lambda^*$ a.s., is exchangeable and extreme among all PES-partitions.
\end{proposition}

\begin{proof}   For $\pi_n\in {\mathcal O}_{\lambda^*}$ the shape $\lambda=s(\pi_n)$ is a composition of $n$ obtained by arranging the parts of $\lambda^*$ in some order. 
Let $Q_{\lambda^*}$ be the set of compositions of $n-1$ obtained from such $\lambda$ by reducing a part by $1$ (and discarding $0$ in case the part was $1$).
For instance, if $\lambda^*=(a,b,c)$ (with $c>1$) , then $Q_{\lambda^*}$ is the set of compositions derived by arranging parts in one of the compositions  $(a-1,b,c), (a,b-1,c), (a,b,c-1)$.

We construct a directed graph on $Q_{\lambda^*}$, by linking entry $\alpha$ to exit $\beta$ if there exists a composition $\lambda$, which is some arrangement of $\lambda^*$ 
derivable from $\alpha$ by increasing a part by $1$ and moving the part to the front, and from $\beta$ by increasing a part by $1$. 
If $\lambda^*$ has parts $1$, the first rule includes the possibility of placing $1$ in the front  of $\alpha$, and the second in the rear of $\beta$.
For example $\alpha=(a,b-1,c)$ is linked to $(b,a,c-1)$ through 
composition $(b,a,c)$.
Interpreting $p$ as a directed flow through the links of $G$, (\ref{PES}) is the balance equation expressing  Kirchhoff's Current Law.

We assert that $G$ is strictly connected.  A chain connecting $\alpha$ to a target composition $(\lambda_1,\ldots,\lambda_k)\in Q_{\lambda^*}$ involves a series of transitions resulting in $(\cdots,\lambda_k)$ 
 or $(\cdots,\lambda_k+1)$. Then the procedure iterates with the last part staying still. For instance, a chain reversing the order in $\alpha=(a-1,b,c)$ has three transitions
$$(a-1,b,c) \stackrel{(a,b,c)}{\longrightarrow} (a,b-1,c)     \stackrel{(b,a,c)}{\longrightarrow} (b, a,c-1)  \stackrel{(c,b,a)}{\longrightarrow} (c, b,a-1).$$
 By the Rank Theorem (\cite{Biggs}, Proposition 4.3), (\ref{StaRule}) has a one-dimensional space of solutions, thus the only probability function satisfying the equation is the symmetric $p_{\lambda^*}^*$.
\end{proof}

Recalling (\ref{symmetris}),  each exchangeable $p_\lambda$ (unless all parts of the ranked shape are equal) decomposes over partially exchangeable partitions. The stationarity as additional 
symmetry feature returns to the single-orbit exchangeable distributions the status of extremes in a larger PES context.

The complexity of the system (\ref{PES}) quickly increases with $n$.
In these equations  the probabilities of trivial partitions $p(1|1|\cdots|1)$, $p(n)$ are uncoupled from the other variables, 
and the  symmetry condition $p(1,n-1)=p(n-1,1)$ keeps the orbit ${\mathcal O}_{(n-1,1)}$.
The next $n=5$ example  shows that  a probability function may have symmetry broken inside every other orbit.

\begin{example}{\rm (PES, nonexchangeable    distribution for $n=5$).  We construct first signed measure with total mass $4$, whose support crosses all nontrivial orbits of ${\mathfrak S}_5$.

\[
\begin{tabular}{@{}c@{\qquad}c@{}}
\begin{tabular}[t]{lc}
\toprule
$s(\Pi_5)$ & $q$ \\
\midrule
(3,2)       & -1 \\
(3,1,1)     & 1 \\
(2,2,1)     & -1 \\
(2,1,1,1)   & 1 \\
(1,2,1,1)   & 2 \\
(1,1,2,1)   & 1 \\
\bottomrule
scale & 4 \\
\end{tabular}
&
\begin{tabular}[t]{lc}
\toprule
$s(\Pi_4)$ & $q$ \\
\midrule
(2,2)       & -2 \\
(2,1,1)     & 1 \\
(1,2,1)     & 1 \\
(1,1,2)     & 1 \\
(1,1,1,1)   & 4 \\
\bottomrule
scale & 4 \\
\end{tabular}
\end{tabular}
\]
Let $\widehat{p}$ be the probability function corresponding to the uniform distribution (\ref{Bell}). 
Since $q\geq -1$ and $B_5=52$, the linear combination
$$p=  \frac{104 \,\widehat{p}+q}   {108}$$
is positive on ${\mathcal P}_5$, and satisfies (\ref{PES}) and normalisation by linearity, hence $p$ is a legitimate PES probability function, 
nonexchangeable due to the asymmetric component.
}
\end{example}
\vskip0.2cm

\begin{example}\label{Ex6}{\rm (The $n=5$ polytope of PES partitions.)
The argument we used for spreadable partitions is applicable here, so we conclude that for $n=5$ the polytope of PES partitions is not simplicial,
and that extreme PES partitions are located on both side of a hyperplane containing the simplex of exchangeable partitions.
The polytope of PES partitions has affine dimension $8$, just one bigger than $7$ for the simplex of exchangeable partitions. An automated search 
returns three non-exchangeable PES partitions $q_1, q_2, q_3$, shown in the table together with $7$ exchangeable extremes.
{\setlength{\tabcolsep}{3pt}
\[
\begin{array}{@{}c|ccc|ccccccc|c@{}}
\lambda
& q_1 & q_2 & q_3
& \multicolumn{7}{c|}{\text{exchangeable probability functions $p_{\lambda^*}$}}
& d(\lambda)\\
\cline{5-11}
& & &
& 5 & 41 & 32 & 311 & 221 & 2111 & 11111 & \\
\hline
(5)
& 0 & 0 & 0
& \mathbf{1} & 0 & 0 & 0 & 0 & 0 & 0 & 1 \\
(1,4)
& 0 & 0 & 0
& 0 & \mathbf{1} & 0 & 0 & 0 & 0 & 0 & 1 \\
(3,2)
& 0 & \mathbf{1} & 0
& 0 & 0 & \mathbf{1} & 0 & 0 & 0 & 0 & 6 \\
(2,3)
& 0 & 0 & \mathbf{1}
& 0 & 0 & \mathbf{1} & 0 & 0 & 0 & 0 & 4 \\
(4,1)
& 0 & 0 & 0
& 0 & \mathbf{1} & 0 & 0 & 0 & 0 & 0 & 4 \\
(3,1,1)
& 0 & 0 & \mathbf{1}
& 0 & 0 & 0 & \mathbf{1} & 0 & 0 & 0 & 6 \\
(1,3,1)
& 0 & \mathbf{1} & 0
& 0 & 0 & 0 & \mathbf{1} & 0 & 0 & 0 & 3 \\
(1,1,3)
& 0 & 0 & \mathbf{1}
& 0 & 0 & 0 & \mathbf{1} & 0 & 0 & 0 & 1 \\
(2,2,1)
& 0 & \mathbf{1} & 0
& 0 & 0 & 0 & 0 & \mathbf{1} & 0 & 0 & 8 \\
(2,1,2)
& \mathbf{1} & 0 & 0
& 0 & 0 & 0 & 0 & \mathbf{1} & 0 & 0 & 4 \\
(1,2,2)
& 0 & 0 & \mathbf{1}
& 0 & 0 & 0 & 0 & \mathbf{1} & 0 & 0 & 3 \\
(2,1,1,1)
& \mathbf{1} & \mathbf{1} & 0
& 0 & 0 & 0 & 0 & 0 & \mathbf{1} & 0 & 4 \\
(1,2,1,1)
& \mathbf{2} & 0 & 0
& 0 & 0 & 0 & 0 & 0 & \mathbf{1} & 0 & 3 \\
(1,1,2,1)
& \mathbf{1} & \mathbf{1} & 0
& 0 & 0 & 0 & 0 & 0 & \mathbf{1} & 0 & 2 \\
(1,1,1,2)
& 0 & \mathbf{2} & 0
& 0 & 0 & 0 & 0 & 0 & \mathbf{1} & 0 & 1 \\
(1,1,1,1,1)
& 0 & 0 & 0
& 0 & 0 & 0 & 0 & 0 & 0 & \mathbf{1} & 1 \\
\hline
\text{scale}
& 16 & 26 & 13
& 1 & 5 & 10 & 10 & 22 & 10 & 1 &
\end{array}
\]
}
By the Pigeonhole Principle the $8$-dimensional PES polytope has a simple structure: $2$  vertices lie on one side of the hyperplane spanned on exchangeable partitions,
and one vertex on the other.
}
\end{example}

\end{document}